\documentclass{amsart}
\usepackage[pagewise]{lineno}
\usepackage{tabularx}
\usepackage{mathrsfs}
\usepackage{amsmath,amssymb,amsfonts,amsthm}
\usepackage{caption}
\usepackage{ctable}
\usepackage{geometry}
\usepackage{rotating}
\usepackage{makecell}
\usepackage{blkarray}
\usepackage{multirow}
\newtheorem{theorem}{Theorem}
\newtheorem{proposition}{Proposition}
\newtheorem{remark}{Remark}

\title[Central Configurations]{On the Degeneracy of the Central Configuration Formed by a Regular n-Gon with a Central Mass}
\author{Tingjie Zhou}
\address{Chern Institute of Mathematics and LPMC, Nankai University,Tianjin, China}
\email{tingjiezhou@nankai.edu.cn}

\author{Zhihong Xia}
\address{Institute for Advanced Research, Great Bay Universtiy, Dongguan, Guangdong, China}
\address{Department of Mathematics, Northwestern University, Evanston, IL, USA}
\email{xia@math.northwestern.edu}
\date{\today}

\begin{document}

\begin{abstract}
We investigate the degeneracy of the central configuration formed by a regular $n$-gon of equal masses together with an additional mass at the center. While degeneracy of such configurations has traditionally been studied through direct spectral computations, a systematic structural understanding of the origin and multiplicity of degeneracy values has remained incomplete. Exploiting the dihedral symmetry $D_n$, we develop a representation-theoretic framework that decomposes the Hessian of $\sqrt{IU}$ into invariant blocks associated with irreducible symmetry modes, reducing the degeneracy problem to a finite collection of low-dimensional determinants. In particular, this decomposition reveals a distinguished $3 \times 3$ block arising from the coupling between the central mass and the first Fourier mode. Within this framework, degeneracy is organized mode by mode: for each admissible Fourier mode $l \geq 2$, there exists at most one critical value of the central mass parameter at which degeneracy occurs, while the mode $l = 1$ exhibits a qualitatively different behavior. As a consequence, all degeneracy values can be determined explicitly, and their number increases with $n$, reflecting the growing number of independent symmetry modes. Our results provide a structural explanation for the multiplicity of degeneracy values and show that degeneracy is not an isolated phenomenon, but a consequence of the underlying symmetry. The approach also suggests a general framework for analyzing degeneracy in symmetric central configurations.

  \textbf{Keywords:} the central $+$ regular $n$-gon configuration, the degeneration, symmetry modes, the dihedral symmetry, matrix blocks
\end{abstract}

\maketitle


\section{Introduction}

We briefly recall the formulation of the $n$-body problem. Let $q_i \in \mathbb{R}^d$ and $m_i>0$ denote the position and mass of the $i$-th body, respectively. The motion is governed by Hamilton’s equations associated with the Hamiltonian
\[
H = \sum_{i=1}^n \frac{|p_i|^2}{2m_i} - U, \qquad 
U = \sum_{1 \le i < j \le n} \frac{m_i m_j}{|q_i - q_j|}.
\]
The moment of inertia is defined by
\[
I = \frac12 \sum_{i=1}^n m_i |q_i|^2.
\]

A configuration is called a \emph{central configuration} if it is a critical point of $\sqrt{I}U$. Such configurations correspond to relative equilibria, in which mutual distances remain constant. A central configuration is \emph{degenerate} if the Hessian of $\sqrt{I}U$ has a nontrivial kernel. Degeneracy is closely related to bifurcations of relative equilibria and changes in stability.

A classical example is the regular $n$-gon configuration formed by equal masses. When an additional mass $m$ is placed at the center, one obtains the \emph{central $+$ regular $n$-gon configuration}, which preserves the full dihedral symmetry. This makes it a natural setting for investigating the interplay between symmetry, spectral structure, and stability.

Palmore~\cite{MR420713} asserted that for $n \ge 4$ there exists a unique value $m^*$ at which degeneracy occurs, and proposed that the regular $n$-gon configuration behaves as a local minimum of $\sqrt{I}U$ under suitable conditions, suggesting a relatively simple parameter dependence of degeneracy. However, subsequent works have revealed a more intricate picture. Meyer and Schmidt~\cite{meyer1988bifurcations} studied bifurcations of relative equilibria for the regular $n$-gon configuration with a central mass and showed that multiple degeneracy values may occur, with increasing complexity as $n$ grows. Slaminka and Woerner~\cite{slaminka1990central} and Woerner~\cite{woerner1990n} further demonstrated that the regular $n$-gon configuration fails to be a local minimum of $\sqrt{I}U$ for $n \ge 6$, contradicting earlier expectations. Subsequent studies have clarified stability, existence, and uniqueness properties of symmetric central configurations.

Despite these advances, a systematic understanding of how degeneracy values arise, how many such values exist, and how they are organized by symmetry remains incomplete. Existing approaches rely primarily on direct spectral computations and do not provide a structural explanation of the multiplicity of degeneracy values, especially in the presence of the additional mass parameter.

The regular $n$-gon is invariant under the dihedral group $D_n$, providing a natural framework for representation-theoretic analysis. Xia~\cite{doi:10.1063/1.2993622,MR4355921} introduced a trace-based method for simplifying the computation of Hessian eigenvalues, but this approach becomes insufficient for determining the full spectrum as $n$ increases. Building on this idea, and exploiting the invariant decomposition induced by the $D_n$-action, we developed in~\cite{Tingjie2025} an effective method for computing the eigenvalues of the Hessian of $\sqrt{I}U$. As a consequence, it was shown that, for sufficiently large $n$, the regular $n$-gon configuration fails to be a local minimum, providing a systematic refinement of earlier instability results.

The main purpose of this paper is to provide a representation-theoretic framework for analyzing degeneracy in the central $+$ regular $n$-gon configuration. Exploiting the $D_n$-action, we decompose the configuration space into invariant subspaces associated with irreducible representations, reducing the Hessian of $\sqrt{I}U$ to block-diagonal form with blocks of dimension at most three.

Within this framework, degeneracy is organized mode by mode. For each Fourier mode $l \ge 2$, the Hessian reduces to a $2\times2$ block depending affinely on the central mass parameter, while the mode $l=1$ produces a distinguished $3\times3$ block due to its coupling with the central mass. As a consequence, each symmetry mode contributes independently to degeneracy.

Our analysis shows that for each admissible mode $l \ge 2$ there exists at most one critical value $m_l^*$ at which degeneracy occurs, while the mode $l=1$ exhibits qualitatively different behavior. In particular, the number of degeneracy values increases with $n$. Moreover, all degeneracy values can be determined explicitly, and their multiplicity admits a clear structural interpretation: each irreducible component gives rise to an independent degeneracy mechanism. This explains the failure of Palmore’s uniqueness claim and clarifies the role of symmetry in shaping the spectrum of the Hessian.

From a dynamical perspective, each degeneracy value corresponds to a potential bifurcation of relative equilibria associated with a specific symmetry mode, thereby identifying distinct instability channels. The approach developed here not only yields explicit results for the central $+$ regular $n$-gon configuration but also provides a general strategy for studying degeneracy in symmetric central configurations.

The paper is organized as follows. Section 2 reviews the representation theory of the dihedral group and derives the invariant decomposition. Section 3 contains the analysis of the Hessian blocks and degeneracy conditions. Technical computations are collected in the Appendix.

\section{Symmetry and invariant decomposition}

In this section, we exploit the dihedral symmetry of the central $+$ regular $n$-gon configuration to derive a canonical decomposition of the configuration space. This decomposition provides the natural framework for reducing the Hessian of $\sqrt{I}U$ to block-diagonal form.

\subsection{The dihedral symmetry}

The dihedral group $D_n$ of order $2n$ is generated by a rotation $r$ and a reflection $s$:
\[
D_n = \langle r, s \mid r^n = s^2 = e,\; s^{-1}rs = r^{-1} \rangle.
\]
It describes the full symmetry group of a regular $n$-gon. The generator $r$ represents a rotation by the angle
\[
\theta = \frac{2\pi}{n},
\]
while $s$ represents a reflection. We denote by
\[
R(\theta)=
\begin{pmatrix}
\cos\theta & -\sin\theta\\
\sin\theta & \cos\theta
\end{pmatrix}
\]
the standard rotation matrix in $\mathbb{R}^2$.

Let
\[
q_k = (\cos k\theta, \sin k\theta), \quad k=1,\dots,n,
\qquad q_{n+1} = (0,0),
\]
and denote by
\[
z_0 = (q_1,\dots,q_{n+1})
\]
the central $+$ regular $n$-gon configuration.

The action of $D_n$ on this configuration induces a linear representation
\[
\mathscr{D}: D_n \to \mathrm{GL}_{2n+2}(\mathbb{R}),
\]
defined by permuting the bodies and applying the corresponding rotation or reflection to each coordinate. More precisely, for each $a\in D_n$,
\[
\mathscr{D}(a)(q_1,\dots,q_n,q_{n+1})
=
\bigl(R_a q_{\sigma_a(1)},\dots,R_a q_{\sigma_a(n)},R_a q_{n+1}\bigr),
\]
where $\sigma_a$ is the induced permutation of the vertices so that $q_{\sigma_a(k)}$ is the point that is mapped to the $k$-th position under the action of $a$ and $R_a\in O(2)$ is the corresponding planar rotation or reflection.

To understand how this symmetry constrains the dynamics and the spectral properties of the Hessian, it is natural to decompose the representation $\mathscr{D}$ into its fundamental symmetry types.
These are described by the irreducible representations of $D_n$.

A systematic introduction to the use of representation-theoretic methods in the $n$-body problem, including symmetry-adapted decompositions, can be found in our previous work~\cite{MR4355921}. For general background on finite group representations, we refer to standard references such as~\cite{FultonHarris1991,serre1977linear}.

For $n$ even, the irreducible representations of $D_n$ consist of four one-dimensional representations $\phi_1,\dots,\phi_4$ and $\frac{n}{2}-1$ two-dimensional representations $\rho_k$. Their explicit forms are summarized in Table~\ref{action}. When $n$ is odd, the representations $\phi_3$ and $\phi_4$ are absent.

\begin{table}[h]
\centering
\renewcommand{\arraystretch}{1.2}
\begin{tabular}{|c|c|c|}
\hline
 & $r^j$ & $r^j s$ \\
\hline
$\phi_1$ & $1$ & $1$\\
\hline
$\phi_2$ & $1$ & $-1$\\
\hline
$\phi_3$ & $(-1)^j$ & $(-1)^j$\\
\hline
$\phi_4$ & $(-1)^j$ & $(-1)^{j+1}$\\
\hline
\makecell{$\rho_k$ \\ $k=1,\dots,\frac{n}{2}-1$}
&
$\begin{pmatrix}
\cos kj\theta & -\sin kj\theta\\
\sin kj\theta & \cos kj\theta
\end{pmatrix}$
&
$\begin{pmatrix}
\cos kj\theta & \sin kj\theta\\
\sin kj\theta & -\cos kj\theta
\end{pmatrix}$ \\
\hline
\end{tabular}
\caption{Irreducible representations of the dihedral group $D_n$ (for $n$ even).}
\label{action}
\end{table}

The representation $\mathscr{D}$ induced by the action on the central $+$ regular $n$-gon configuration decomposes as
\[
\mathscr{D} \sim
\phi_1 \oplus \phi_2 \oplus \phi_3 \oplus \phi_4
\oplus 3\rho_1 \oplus 2\rho_2 \oplus \cdots \oplus 2\rho_{\frac{n}{2}-1}
\]
for $n$ even, and
\[
\mathscr{D} \sim
\phi_1 \oplus \phi_2
\oplus 3\rho_1 \oplus \cdots \oplus 2\rho_{\frac{n-1}{2}}
\]
for $n$ odd. The multiplicities follow from standard character computations, which we omit.

\subsection{Equivariance of the Hessian and isotypic decomposition}

Let $f \in C^2(\mathbb{R}^{2n+2})$ be invariant under the $D_n$-action:
\[
f(\mathscr{D}(a)z)=f(z), \qquad \forall a\in D_n.
\]
If $z_0$ is fixed by the action, that is,
\[
\mathscr{D}(a)z_0=z_0, \qquad \forall a\in D_n,
\]
then differentiation yields
\[
\mathscr{D}(a) D^2f(z_0)=D^2f(z_0)\mathscr{D}(a), \qquad \forall a\in D_n.
\]
Thus the Hessian belongs to the commutant of the representation $\mathscr{D}$. By Maschke's theorem~\cite{FultonHarris1991,serre1977linear}, the representation space admits a decomposition into isotypic components
\[
\mathbb{R}^{2n+2}=\bigoplus_{\lambda} V^{(\lambda)},
\]
where each $V^{(\lambda)}$ is the sum of all subrepresentations isomorphic to a fixed irreducible representation. Since $D^2f(z_0)$ commutes with the group action, it maps each isotypic component into itself as a consequence of Schur's lemma~\cite{FultonHarris1991,serre1977linear}.

\begin{remark}
The isotypic decomposition provides a coarse invariant splitting of the configuration space. However, it is not yet fine enough for the block-diagonalization of the Hessian, since equivalent irreducible summands within the same isotypic component may still be coupled. To obtain the minimal invariant subspaces relevant for computation, one needs a further refinement adapted to the generators $r$ and $s$.
\end{remark}

\subsection{Fourier-type decomposition}

We now refine the isotypic decomposition by exploiting the rotational symmetry.

The operator $\mathscr{D}(r)$ acts as a cyclic shift combined with a planar rotation. Its complex eigenvalues are of the form
\[
e^{\pm il\theta}, \qquad l=0,1,\dots,\Bigl\lfloor \frac{n}{2}\Bigr\rfloor,
\]
which leads naturally to a Fourier decomposition along the polygon.

For each $l=0,1,\dots,\lfloor n/2\rfloor$, define the complex vectors
\[
v_1^{l\theta,r}
=
\bigl(
e^{-il\theta}(\cos\theta,\sin\theta),\;
e^{-i2l\theta}(\cos2\theta,\sin2\theta),\;
\dots,\;
e^{-inl\theta}(\cos n\theta,\sin n\theta),\;
0
\bigr),
\]
and
\[
v_2^{l\theta,r}
=
\bigl(
e^{-il\theta}(-\sin\theta,\cos\theta),\;
e^{-i2l\theta}(-\sin2\theta,\cos2\theta),\;
\dots,\;
e^{-inl\theta}(-\sin n\theta,\cos n\theta),\;
0
\bigr).
\]

For $l\ge2$, the vectors $v_1^{l\theta,r}$ and $v_2^{l\theta,r}$ form a basis of the eigenspace
\[
E^r_{-l\theta}
\]
associated with the eigenvalue $e^{-il\theta}$.

The case $l=1$ is exceptional, since the central mass directions also transform under the rotation block $R(\theta)$. Defining
\[
\eta_- = e_{2n+1} + i e_{2n+2}, \qquad
\eta_+ = e_{2n+1} - i e_{2n+2},
\]
one has
\[
\mathscr{D}(r)\eta_- = e^{-i\theta}\eta_-, \qquad
\mathscr{D}(r)\eta_+ = e^{i\theta}\eta_+,
\]
and therefore
\[
E^r_{-\theta}
=
\operatorname{span}_{\mathbb C}\{v_1^{\theta,r},\, v_2^{\theta,r},\, \eta_-\},
\qquad
E^r_{\theta}
=
\operatorname{span}_{\mathbb C}\{\overline{v_1^{\theta,r}},\, \overline{v_2^{\theta,r}},\, \eta_+\}.
\]

Since the representation is real, the conjugate eigenspaces
\[
E^r_{l\theta} \oplus E^r_{-l\theta}
\]
combine to form real invariant subspaces. For $l\ge2$, this yields a real four-dimensional space, while for $l=1$ one obtains a real six-dimensional space, reflecting the additional contribution of the central mass.

Moreover, each isotypic component corresponding to a irreducible representation is realized in this way; more precisely,
\[
V^{(l)} \simeq E^r_{l\theta} \oplus E^r_{-l\theta}.
\]

\medskip

\subsection{Refinement via reflection symmetry}

To obtain subspaces invariant under the full dihedral group, we further decompose the above spaces using the reflection operator $\mathscr{D}(s)$.

Let $E^s_{\pm1}$ denote the eigenspaces of $\mathscr{D}(s)$. Then the intersections
\[
E^s_{\mu} \cap \bigl(E^r_{l\theta} \oplus E^r_{-l\theta}\bigr), \qquad \mu=\pm1,
\]
are invariant under both generators, and hence under the full group $D_n$.

For $l \neq 1$, one computes
\begin{align*}
E^s_{1} \cap \bigl(E^r_{l\theta}\oplus E^r_{-l\theta}\bigr)
&=
\operatorname{span}\bigl\{
\operatorname{Re}(v_1^{l\theta,r}),\;
\operatorname{Im}(v_2^{l\theta,r})
\bigr\}, \\
E^s_{-1} \cap \bigl(E^r_{l\theta}\oplus E^r_{-l\theta}\bigr)
&=
\operatorname{span}\bigl\{
\operatorname{Re}(v_2^{l\theta,r}),\;
\operatorname{Im}(v_1^{l\theta,r})
\bigr\}.
\end{align*}
Accordingly, we introduce the real vectors
\[
v_l = \operatorname{Re}(v_1^{l\theta,r}), \qquad
v_l' = \operatorname{Im}(v_2^{l\theta,r}),
\]
\[
w_l = \operatorname{Re}(v_2^{l\theta,r}), \qquad
w_l' = -\operatorname{Im}(v_1^{l\theta,r}),
\]
so that $\{v_l,v_l'\}$ and $\{w_l',w_l\}$ form bases of two equivalent $D_n$-invariant subspaces. These two-dimensional subspaces constitute the \emph{symmetry modes} for $l\ge2$.

For $l=1$, the situation is different. The space
\[
E^r_{\theta} \oplus E^r_{-\theta}
\]
is six-dimensional, and after refinement by the reflection symmetry, one obtains a three-dimensional invariant subspace generated by the first Fourier mode together with the central mass directions. This gives rise to the unique $3\times3$ block in the Hessian.

\medskip

\subsection{Block structure of the Hessian}

The above construction yields a decomposition of the configuration space into symmetry modes invariant under the Hessian.

\begin{theorem}[Block structure]
In a symmetry-adapted basis, the Hessian $D^2f(z_0)$ is block-diagonal. The blocks correspond to symmetry modes and have the following structure:
\begin{itemize}
\item scalar blocks corresponding to one-dimensional modes,
\item $2\times2$ blocks corresponding to modes with $l\ge2$,
\item a $3\times3$ block associated with the mode $l=1$.
\end{itemize}
\end{theorem}

\begin{remark}[Coupling at $l=1$]
For $l\ge2$, the Fourier modes remain decoupled from the central mass, leading to $2\times2$ blocks. In contrast, the central mass transforms according to the same representation as the first Fourier mode. This coincidence allows the Hessian to mix these directions, enlarging the invariant subspace and producing the $3\times3$ block.
\end{remark}

\begin{remark}[Parity of $n$]
The number of scalar blocks depends on the parity of $n$. When $n$ is even, there are four scalar blocks corresponding to the one-dimensional representations $\phi_1,\dots,\phi_4$, while for $n$ odd only two such blocks appear, corresponding to $\phi_1$ and $\phi_2$.
\end{remark}

\section{Degeneracy of the central $+$ regular $n$-gon configuration}

In this section, we analyze the degeneracy of the central $+$ regular $n$-gon configuration using the symmetry-adapted decomposition obtained in Section~2. This reduces the problem to the study of a finite collection of low-dimensional matrices. We begin by summarizing the main conclusions of this section. 

\begin{theorem}[Main Theorem]\label{thm:main}
Consider the central $+$ regular $n$-gon configuration with equal masses on the vertices and a central mass $m>0$.

\begin{enumerate}
\item[\textup{(1)}] (\emph{Block structure})
In a symmetry-adapted basis associated with the dihedral group $D_n$, the Hessian $D^2f(z_0)$ decomposes into invariant blocks corresponding to symmetry modes:
\begin{itemize}
\item scalar blocks associated with one-dimensional modes,
\item $2\times2$ blocks associated with Fourier modes $l\ge2$,
\item a distinguished $3\times3$ block associated with the mode $l=1$.
\end{itemize}
\item[\textup{(2)}] (\emph{Degeneracy criterion})
Let $A_1$ and $A_l$ ($l\ge2$) denote the symmetry-adapted blocks of the Hessian $D^2f(z_0)$ associated with the decomposition above.
The configuration $z_0$ is degenerate if the Hessian $D^2f(z_0)$ if and only if
\[
\det A_1 = 0 \quad \text{or} \quad \det A_l = 0 \quad \text{for some } l \ge 2.
\]
\item[\textup{(3)}] (\emph{Mode-wise degeneracy for $l\ge2$})
For each admissible Fourier mode $l\ge2$, there exists at most one positive value of the central mass parameter $m$ at which degeneracy occurs.

\item[\textup{(4)}] (\emph{Exceptional behavior of the mode $l=1$})
The mode $l=1$ behaves differently due to its coupling with the central mass. In this case, degeneracy occurs for a unique positive value of $m$ when $3\le n\le6$, and does not occur for $n\ge7$.

\item[\textup{(5)}] (\emph{Multiplicity of degeneracy values})
The total number of positive degeneracy values is determined by the admissible symmetry modes. In particular, for $n\ge7$, this number increases with $n$.
\end{enumerate}
\end{theorem}

\medskip

\noindent
The theorem shows that degeneracy is organized mode by mode: each symmetry mode contributes an independent mechanism, while the exceptional role of the first mode arises from its coupling with the central mass.  As a consequence of this decomposition, all degeneracy values can be determined explicitly, and their total number is completely characterized.

The proof is obtained by analyzing the symmetry-adapted blocks introduced in Section~2.

\subsection{Notation}

Let
\[
I_0=\sqrt{2I(z_0)}=\sqrt n, \qquad
d_{kj}=|q_k-q_j|=2\sin\frac{|k-j|\theta}{2}, \quad \theta=\frac{2\pi}{n},
\]
and
\[
d_0=\sum_{k=1}^{n-1}\frac1{d_{nk}}, \qquad
U_0 = U(z_0), \qquad U_0 = U_{e0} + nm.
\]

For $l=1,\dots,\lfloor n/2\rfloor$, define
\[
\binom{U_{l1}}{U_{l2}}
=
\sum_{k=1}^{n-1}\frac1{2d_{nk}^3}
\binom{1-\cos k\theta \cos lk\theta-3(\cos k\theta-\cos lk\theta)}
{i\sin k\theta\sin lk\theta},
\]
\[
\binom{U'_{l1}}{U'_{l2}}
=
\sum_{k=1}^{n-1}\frac1{2d_{nk}^3}
\binom{-i\sin k\theta\sin lk\theta}
{1-\cos k\theta \cos lk\theta+3(\cos k\theta-\cos lk\theta)}.
\]

\subsection{Computation of the blocks}

We now describe how the block representation of the Hessian can be computed explicitly in the symmetry-adapted basis introduced in Section~2. The full details are given in the Appendix.

By the decomposition into symmetry modes, the configuration space splits into invariant subspaces. In the corresponding basis, the Hessian $H=D^2f(z_0)$ admits a block-diagonal form
\[
H \sim \mathrm{diag}(\lambda_1,\lambda_2,\dots,\lambda_s, A_1,B_1,\dots,A_l,B_l,\dots),
\]
where $s$ denotes the number of scalar blocks. More precisely, $s=4$ if $n$ is even and $s=2$ if $n$ is odd.

Let
\[
H = D^2 f(z_0)
=
\bigl(H_{kj}'\bigr)_{1\le k,j\le n+1},
\qquad
H_{kj}' \in \mathbb{R}^{2\times2},
\]
where each $H_{kj}'$ represents the interaction between the $k$-th and $j$-th bodies. By rotational symmetry,
\[
\mathscr{D}(r) H \mathscr{D}^{T}(r)=H,
\]
which implies
\[
H_{j,\,j+k}'=R(j\theta)\,H_{nk}'\,R(j\theta)^T.
\]

\begin{proposition}
For each symmetry mode indexed by $l \geq 2$, the action of the Hessian on the corresponding invariant subspace can be expressed as
\[
H v_1^{l\theta,r}
=
h_{l1}v_1^{l\theta,r}
+
h_{l2}v_2^{l\theta,r},
\qquad
H v_2^{l\theta,r}
=
h_{l1}'v_1^{l\theta,r}
+
h_{l2}'v_2^{l\theta,r},
\]
where
\[
\begin{pmatrix}
h_{l1}\\ h_{l2}
\end{pmatrix}
=
\sum_{k=1}^{n}
H_{nk}'\,e^{-ilk\theta}
(\cos k\theta,\sin k\theta)^T,
\]
\[
\begin{pmatrix}
h_{l1}'\\ h_{l2}'
\end{pmatrix}
=
\sum_{k=1}^{n}
H_{nk}'\,e^{-ilk\theta}
(-\sin k\theta,\cos k\theta)^T.
\]
\end{proposition}

\begin{proof}
Using the equivariance relation
\[
H_{j,\,j+k}' = R(j\theta)\, H_{nk}'\, R(j\theta)^T,
\]
the action of $H$ on each symmetry mode reduces to a convolution-type sum over $k$. This yields a discrete Fourier transform structure, and substituting the expressions of the basis vectors $v_1^{l\theta,r}$ and $v_2^{l\theta,r}$ gives the stated formulas.
\end{proof}

In the real bases $\{v_l,v_l'\}$ and $\{w_l',w_l\}$ introduced in Section~2, the corresponding Hessian blocks take the form
\[
A_l = B_l =
\begin{pmatrix}
h_{l1} & -i h_{l1}'\\
i h_{l2} & h_{l2}'
\end{pmatrix}.
\]
Although the above expression involves complex entries, the coefficients $h_{l1},h_{l2},h_{l1}',h_{l2}'$ arise from Fourier-type sums and may be complex-valued. However, the resulting matrix $A_l$ is real when expressed in the underlying real basis, and therefore represents a real linear operator on the invariant subspace.

\subsection{Structure of the blocks}

\begin{proposition}[Structure of $A_l$ for $l\ge2$]
For $l\ge2$, the block $A_l$ has the form
\[
A_l =
\begin{pmatrix}
\frac{U_0}{I_0} + I_0(U_{l1}+2m) & -i I_0 U'_{l1} \\
i I_0 U_{l2} & \frac{U_0}{I_0} + I_0(U'_{l2}-m)
\end{pmatrix}.
\]
\end{proposition}

\begin{proposition}[Structure of $A_1$]
The block $A_1$ is a $3\times3$ matrix given by
\[
A_1 =
\begin{pmatrix}
\frac{U_0}{I_0} + I_0(U_{11}+2m) & -iI_0 U'_{11} & -2I_0 m \\
iI_0 U_{12} & \frac{U_0}{I_0} + I_0(U'_{12}-m) & I_0 m \\
- I_0 n m & I_0 n m & \frac{2U_0 m}{I_0} + \frac{I_0 n m}{2}
\end{pmatrix}.
\]
\end{proposition}

\begin{proposition}\label{prop:scalar}
Among the scalar blocks, two correspond to translational invariance and satisfy
\[
\lambda_1=\lambda_2=0.
\]
When $n$ is even, the remaining scalar blocks satisfy
\[
\lambda_3\lambda_4 = \det A_{n/2},
\]
so that their contribution is already encoded in the block $A_{n/2}$.
\end{proposition}

The configuration $z_0$ is degenerate if the Hessian $H$ has a nontrivial kernel. Since $H$ is block-diagonal in the symmetry-adapted basis, this occurs if and only if at least one of the blocks is singular.

\begin{proposition}[Degeneracy criterion]
The configuration $z_0$ is nontrivially degenerate if and only if
\[
\det A_1=0
\quad \text{or} \quad
\det A_l=0 \text{ for some } l\ge2.
\]
\end{proposition}

\subsection{Degeneracy}
For $l\ge2$, one computes
\[
\det A_l = a_l + 3m I_0\Big(\frac{U_{e0}}{I_0}+I_0U'_{l2}\Big),
\]
where
\[
a_l=
\Big(\frac{U_{e0}}{I_0}+I_0U_{l1}\Big)
\Big(\frac{U_{e0}}{I_0}+I_0U'_{l2}\Big)
-I_0^2U'_{l1}U_{l2}.
\]

To simplify the sign analysis, introduce
\[
\beta_l=
\frac{4}{d_0^2}
\left[
\Big(\frac{d_0}{2}+U_{l1}\Big)\Big(\frac{d_0}{2}+U'_{l2}\Big)-U'_{l1}U_{l2}
\right].
\]
Then $\operatorname{sign}(a_l)=\operatorname{sign}(\beta_l)$. By Lemma~3.4 of Moeckel~\cite{MR1350320}, one has $\beta_l<0$ for all admissible $l$, except for the case $l=2$ with $4\le n\le9$, where $\beta_2>0$.

\begin{theorem}[Critical values for $l\ge2$]
For each admissible $l\ge2$ satisfying
\[
a_l\Big(\frac{U_{e0}}{I_0}+I_0U'_{l2}\Big)<0,
\]
there exists a unique positive value
\[
m_l^*=-\frac{a_l}{3I_0\left(\frac{U_{e0}}{I_0}+I_0U'_{l2}\right)}
\]
such that $\det A_l=0$.
\end{theorem}

\begin{proof}
The determinant $\det A_l$ is affine in $m$. If the coefficient of $m$ and the constant term have opposite signs, there exists a unique positive root. The sign of $a_l$ follows from the previous proposition and Moeckel's result.
\end{proof}

The degeneracy condition $\det A_1=0$ reduces to a quadratic equation
\[
\det A_1=b_n m^2+c_n m+d_n,
\]
with coefficients
\[
b_n=3I_0\left(\frac{I_0^3}{2}-\frac{U_{e0}}{I_0}-I_0U_{11}\right),
\]
\[
c_n=I_0^2U_{e0}-\frac{4U_{e0}^2}{I_0^2}-I_0^4U_{11}-5U_{e0}U_{11},
\]
\[
d_n=-\left(\frac{I_0^2}{2}+\frac{U_{e0}}{I_0^2}\right)
\left(\frac{U_{e0}^2}{I_0^2}+2U_{e0}U_{11}\right).
\]

\begin{proposition}
The equation $\det A_1=0$ admits:
\begin{itemize}
\item a unique positive solution for $3\le n\le6$,
\item no positive solution for $n\ge7$.
\end{itemize}
\end{proposition}

\begin{proof}
Let $p(m)=b_n m^2+c_n m+d_n$. The sign properties
\[
c_n<0,\qquad d_n<0,
\]
and
\[
b_n>0 \text{ for } 3\le n\le6, \qquad b_n<0 \text{ for } n\ge7,
\]
follow from the explicit expressions together with the positivity of $U_{e0}$ and known estimates for $U_{11}$.

The conclusion follows from the behavior of $p(m)$ on $[0,+\infty)$.
\end{proof}

In particular, for the $3+1$ body problem, following the method introduced, we have
\[
m^* = \frac{2\sqrt{3} + 9}{18\sqrt{3} - 15},
\]
in agreement with previous results~\cite{MR420713}.

\begin{theorem}[Multiplicity of degeneracy values]
For $n\ge3$, the number of distinct critical values $m^*$ is given by Table~\ref{tab:degeneracy}. In particular, for $n\ge 7$, this number increases with $n$.
\end{theorem}

\begin{proof}
For each admissible $l\ge2$, Theorem above yields a critical value $m_l^*$ except in the exceptional case. The contribution of the $l=1$ mode is determined by the previous proposition. Combining these contributions yields the result.
\end{proof}

\begin{table}[htbp]
\centering
\caption{Number of critical values $m^*$ corresponding to nontrivial degeneracy}
\label{tab:degeneracy}
\begin{tabular}{c|c|c|c|c}
\hline
 & 3
 & $4 \le n \le 6$
 & $7 \le n \le 9$
 & $n \ge 10$ \\
\hline
 $\#\, m^*$ &
 $1$ & 
 $\lfloor \frac{n}{2}\rfloor -1  $ &
 $\lfloor \frac{n}{2} \rfloor -2$ &
 $\lfloor \frac{n}{2}\rfloor -1  $ \\
\hline
\end{tabular}
\end{table}



\bibliographystyle{acm}
\bibliography{ref}


\newpage
\appendix

\section{Computation of the Hessian blocks}

The purpose of this appendix is to provide the detailed computations leading to the block representation of the Hessian described in Section~3.

\subsection{Decomposition of the Hessian}

Recall that
\[
D^2 f(z_0)
=
\Bigl[
\nabla(\sqrt{2I}) \nabla U^{T}
+ \nabla U \nabla(\sqrt{2I})^{T}
+ U D^2(\sqrt{2I})
+ \sqrt{2I} D^2 U
\Bigr](z_0).
\]

We write
\[
\nabla(\sqrt{2I}) \nabla U^{T}(z_0) = (C_{kj}), \quad
D^2(\sqrt{2I})(z_0) = (D_{kj}), \quad
D^2 U(z_0) = (V_{kj}),
\]
where $C_{kj}, D_{kj}, V_{kj} \in \mathbb{R}^{2\times 2}$.

All matrices are computed in the symmetry-adapted bases introduced in Section~2.

\subsection{The matrices $\nabla(\sqrt{2I}) \nabla U^{T}$ and $\nabla U \nabla(\sqrt{2I})^{T}$}

For $k \neq n+1$, we have
\[
C_{nk}
=
\sum_{\substack{1 \le j \le n \\ j \neq k}}
\frac{q_n}{I_0} \frac{u_{kj}^{T}}{d_{jk}^2}
-
\frac{m}{I_0} q_n q_k^{T},
\qquad
C_{(n+1)k} = 0,
\qquad
C_{(n+1)(n+1)} = 0.
\]

Define
\[
C_{\rho_l}
=
\sum_{k=1}^{n}
C_{nk} e^{-i l k \theta}
(\cos k\theta, \sin k\theta)^{T},
\]
\[
C_{\rho_l}'
=
\sum_{k=1}^{n}
C_{nk} e^{-i l k \theta}
(-\sin k\theta, \cos k\theta)^{T}.
\]

A direct computation shows that
\[
C_{\rho_l} = C_{\rho_l}' = 0, \qquad \text{for all } l \neq 0.
\]
For $l = 0$, the first component of $C_{\rho_0}$ equals $ 
-\frac{I_0 d_0}{2} - I_0 m,$
while the second component of $C_{\rho_0}'$ vanishes. 
The vector $v_0$ is an eigenvector with eigenvalue $-\frac{I_0 d_0}{2} - I_0 m$, whereas
$w_0$, $v_{\frac{n}{2}}$, and $w_{\frac{n}{2}}$ correspond to the eigenvalue $0$.

For all symmetry modes $l \ge 2$, the contribution of
\[
\nabla(\sqrt{2I}) \nabla U^{T}
\]
vanishes. The same conclusion holds for $\nabla U \nabla(\sqrt{2I})^{T}$.

\subsection{The matrix $D^2(\sqrt{2I})(z_0)$}

The entries of $D^2(\sqrt{2I})(z_0)$ are given by
\[
D_{nk} = -\frac{q_n q_k^{T}}{I_0^3}, \quad
D_{nn} = -\frac{q_n q_n^{T}}{I_0^3} + \frac{E_2}{I_0},
\]
\[
D_{(n+1)k} = 0, \quad
D_{(n+1)(n+1)} = \frac{m}{I_0} E_2.
\]

Define
\[
I_{\rho_l}
=
\sum_{k=1}^{n}
D_{nk} e^{-i l k \theta}
(\cos k\theta, \sin k\theta)^{T},
\]
\[
I_{\rho_l}'
=
\sum_{k=1}^{n}
D_{nk} e^{-i l k \theta}
(-\sin k\theta, \cos k\theta)^{T}.
\]

For $l=1,\dots,\lfloor n/2\rfloor$, one obtains
\[
I_{\rho_l}
=
\begin{pmatrix}
\frac{1}{I_0} \\[0.3ex]
0
\end{pmatrix},
\qquad
I_{\rho_l}' = I_{\rho_0}'
=
\begin{pmatrix}
0 \\[0.3ex]
\frac{1}{I_0}
\end{pmatrix},
\]
while $I_{\rho_0}=0$. The vector $v_0$ corresponds to eigenvalue $0$, whereas
$w_0'$, $v_{\frac{n}{2}}$, and $w_{\frac{n}{2}}'$
correspond to eigenvalue $\frac{1}{I_0}$. 

On each symmetry mode with $l\ge2$, the matrix $D^2(\sqrt{2I})(z_0)$ reduces to
\[
\begin{pmatrix}
\frac{1}{I_0} & 0 \\
0 & \frac{1}{I_0}
\end{pmatrix}.
\]

\subsection{The matrix $D^2 U(z_0)$}

The Hessian of the potential satisfies
\[
V_{nk} = d_{nk}^{-3}(E_2 - 3 u_{nk} u_{nk}^{T}), \quad k \neq n,n+1,
\]
\[
V_{(n+1)k} = m(E_2 - 3 q_k q_k^{T}), \quad k \neq n+1,
\]
\[
V_{nn} = -\sum_{k=1}^{n-1} V_{nk} - V_{(n+1)n}.
\]

Define
\[
U_{\rho_l}
=
\sum_{k=1}^{n}
V_{nk} e^{-i l k \theta}
(\cos k\theta, \sin k\theta)^{T},
\]
\[
U_{\rho_l}'
=
\sum_{k=1}^{n}
V_{nk} e^{-i l k \theta}
(-\sin k\theta, \cos k\theta)^{T}.
\]

A direct computation yields
\[
U_{\rho_l}
=
\begin{pmatrix}
U_{l1} + 2m \\
U_{l2}
\end{pmatrix},
\qquad
U_{\rho_l}'
=
\begin{pmatrix}
U_{l1}' \\
U_{l2}' - m
\end{pmatrix}.
\]
The eigenvalues associated with
$v_0$, $w_0$, $v_{\frac{n}{2}}$, and $w_{\frac{n}{2}}$
are
$U_{01}+2m$, $U_{02}'-m$,
$U_{\frac{n}{2}1}+2m$, and $U_{\frac{n}{2}2}'-m$,
respectively.

Therefore, on each symmetry mode $l\ge2$, the matrix $D^2 U(z_0)$ is represented by
\[
\begin{pmatrix}
U_{l1}+2m & -i U_{l1}' \\
i U_{l2} & U_{l2}'-m
\end{pmatrix}.
\]

\subsection{The $l=1$ mode}

Finally, we consider the interaction with the central mass. One verifies that
\begin{align*}
D^2U(z_0) e_{2n+1}
&= m\bigl(-2 v_1 + v_1' + \tfrac{n}{2} e_{2n+1}\bigr), \\
D^2U(z_0) e_{2n+2}
&= m\bigl(-2 w_1 + w_1' + \tfrac{n}{2} e_{2n+2}\bigr),
\end{align*}
while
\[
D^2(\sqrt{2I})(z_0)e_{2n+1}
=
\frac{m}{I_0} e_{2n+1},
\qquad
D^2(\sqrt{2I})(z_0)e_{2n+2}
=
\frac{m}{I_0} e_{2n+2}.
\]

Moreover, the matrices
\[
\nabla(\sqrt{2I}) \nabla U^{T}(z_0), \quad
\nabla U \nabla(\sqrt{2I})^{T}(z_0)
\]
vanish on the central mass directions.

For the terms involving $\nabla(\sqrt{2I}) \nabla U^{T}$, $D^2(\sqrt{2I})$, and $\nabla U \nabla(\sqrt{2I})^{T}$, a direct computation shows that the first $n$ block components in the last row vanish at $z_0$. As a consequence, when restricting to the subspace generated by the first Fourier mode, the action of these terms on the polygonal components coincides with the case $l\ge2$.

However, for the term $D^2 U(z_0)$, the contribution of the $(n+1)$-th component (corresponding to the central mass) must be taken into account. This produces additional terms in the directions $e_{2n+1}$ and $e_{2n+2}$, which yield the coupling between the central mass and the first Fourier mode.

Combining these contributions yields the explicit form of the $3\times3$ blocks $A_1$ and $B_1$ given in Section~3.

\end{document}